\documentclass[amssymb,11pt]{article}
\usepackage{amsfonts}
\usepackage{amsthm}
\usepackage{amsmath}
\usepackage{graphicx}
\usepackage[psamsfonts]{amssymb}
\usepackage{amscd}

\title{The centralizer of a $C^1$ generic diffeomorphism is trivial}
\author{Christian Bonatti, Sylvain Crovisier and Amie Wilkinson}

 \def\RR{{\mathbb R}}  
   
 \def\ZZ{{\mathbb Z}}

\def\cB{{\cal B}}    \def\cT{{\cal T}}
   \def\cO{{\cal O}} 
\def\cD{{\cal D}}    
  \def\cK{{\cal K}}  
   \def\cR{{\cal R}} \def\cX{{\cal X}}

\newtheorem{mainthm}{Theorem}

\newtheorem{ques}{Question}

\newtheorem{theorem}{Theorem}[section]
\newtheorem{theo}[theorem]{Theorem}
\newtheorem{prop}[theorem]{Proposition}
\newtheorem{proposition}[theorem]{Proposition}

\newtheorem{coro}[theorem]{Corollary}

\theoremstyle{remark}
\newtheorem{remark}[theorem]{Remark}
\newtheorem{remarks}[theorem]{Remarks}
\newtheorem*{remarkempty}{Remark}

\newtheorem*{thank}{Acknowledgement}

\def\dim{\hbox{dim} }
\def\interior{\hbox{Int} }

\def\Lip{\hbox{Lip} }

\def\Diff{\hbox{Diff} }
\def\diff{\hbox{Diff} }
\def\Det{\hbox{Det} }
\def\Per{\hbox{Per} }

\def\title{\em}

\def\M{\mathcal{M}}

\def\transverse{\,\raise2pt\hbox to1em{\hfil$\top$\hfil}\hskip -1em \hbox
to1em{\hfil$\cap$\hfil}\,}

\newlength{\figboxwidth} 

\begin{document}

\maketitle

\section*{Introduction}

In this announcement, we describe the solution in the $C^1$ topology to a question asked by S. Smale on the
genericity of trivial centralizers.  
The question is posed in the following context.
We fix a compact connected manifold $M$ and consider the space $\Diff^r(M)$ of  $C^r$ diffeomorphisms
of $M$, endowed with the $C^r$ topology.  The
{\em centralizer} of  $f\in \Diff^r(M)$ is defined
as $$Z^r(f):=\{g\in \Diff^r(M): fg=gf\}.$$
Clearly $Z^r(f)$ always contains the cyclic group
$<f>$ of all the powers of $f$.  We
say that $f$ has {\em trivial centralizer} if $Z^r(f) = <f>$.
Smale asked the following:

\begin{ques}[\cite{Sm1,Sm2}]\label{q.smale}
Consider the set of $C^r$ diffeomorphisms of a compact connected manifold $M$ with trivial centralizer.
\begin{enumerate}
\item Is this set dense in $\Diff^r(M)$?
\item Is it {\em residual} in $\Diff^r(M)$?  That is, does it contain a dense $G_\delta$ subset?
\item Does it contain an open and dense subset of $\Diff^r(M)$?
\end{enumerate}
\end{ques}

For the case $r=1$ we now have a complete answer to this question.

\begin{mainthm}[B-C-W]\label{t.main1}
For any compact connected manifold $M$, there is a residual subset of $\Diff^1(M)$ consisting
of diffeomorphisms with trivial centralizer.
\end{mainthm}

\begin{mainthm}[B-C-Vago-W]\label{t.main2}
For any compact manifold $M$, the set of $C^1$ diffeomorphisms
with trivial centralizer does not contain any open and dense subset.
\end{mainthm}

Theorem~\ref{t.main1} gives an affirmative answer to the second (and
hence the first) part of Question~\ref{q.smale}: our aim in this text is
to present the structure of its proof that will be detailed in~\cite{BCW2}.
Theorem~\ref{t.main2} gives a negative answer to the third part of
Questions~\ref{q.smale}: with G. Vago we prove in~\cite{BCVW} that there exists
a family of $C^\infty$ diffeomorphisms with large centralizer that is $C^1$ dense in a
nonempty open subset of $\Diff^1(M)$. For these examples, one has to consider separately
the case of the circle, the surfaces and manifolds of dimension greater or equal to $3$:
in dimension less or equal to two such a diffeomorphism appears as the time-1 map of a flow,
whereas in higher dimension each example we build possesses an open set of periodic points.

These results suggest that the topology of the set of diffeomorphisms with trivial centralizer is complicated and motivate
the following questions.
\begin{ques}
\begin{enumerate}
\item {\em Consider the set of diffeomorphisms whose centralizer is trivial.}\\
What is its interior?
\item Is it a Borel set?\\
{\em (See, \cite{FRW} for a negative answer to this question in the measurable context.)}
\item {\em The set $\{(f,g) \in \Diff^1(M)\times \Diff^1(M): fg=gf\}$ is closed}.\\
What is its local topology? For example, is
it locally connected?
\end{enumerate}
\end{ques}

Our motivation for considering Question~\ref{q.smale} comes from at
least two sources. First, the study of $C^1$-generic diffeomorphisms
has seen substantial progress in the last decade, and
Question~\ref{q.smale} is an elementary test question for the
existing techniques.  More intrinsically, there are several
classical motivations for Question~\ref{q.smale}.  In physics (for
example, in Hamiltonian mechanics), one searches for symmetries of a
given system in order to reduce the complexity of the orbit space.
The groups of such symmetries is precisely the centralizer.   In a
similarly general vein, a central theme in dynamics is to understand
the conjugacy classes inside of $\Diff^r(M)$; that is, to find the
orbits of the action of $\Diff^r(M)$ on itself by conjugacy.
Theorem~\ref{t.main1} implies that the stabilizer in this action of
a generic element is trivial.

Knowing the centralizer of a diffeomorphism gives answers to more
concrete questions as well, such as the embeddability of a
diffeomorphism in a flow and the existence of roots of a
diffeomorphism. The study of diffeomorphisms and flows are closely
related, and indeed, every diffeomorphism appears as the return map
of a smooth flow to a cross-section, and the time-1 map of a flow is
a diffeomorphism.  These two studies have many differences as well,
and it is natural to ask when a given diffeomorphism can be embedded
as the time-1 map of a flow ( the centralizer of such a
diffeomorphism must contain either $\RR$ or the circle $\RR/\ZZ$). A
weaker question is to ask whether a diffeomorphism $f$ admits a
root; that is, if one can write $f=g^k$, for some integer $k>1$.  If
$f$ admits such a root, then its centralizer is not trivial,
although it might still be discrete.

Question~\ref{q.smale} can also be viewed as a problem
about the group structure of $\Diff^1(M)$, from a generic vantage point. 
An easy transversality argument (written in~\cite[Proposition 4.5]{ghys} for circle homeomorphisms)
allows to describe the group generated by a generic family of diffeomorphisms:
\emph{for a generic $(f_1, \ldots, f_p) \in \left(\Diff^r(M)\right)^p$ with $p\geq 2$ and $r\geq 0$,
the group $<f_1, \ldots, f_p>$ is free.} Restated in these terms,
Theorem~\ref{t.main1} says that for a generic $f$, if $G$ is any abelian subgroup of
$\Diff^1(M)$ containing $f$, then $G = <f>$.  The same conclusion holds if $G$ is assumed to be nilpotent, for then
the center of $G$, and thus $G$ itself, must equal $<f>$.  One can ask whether the same conclusions hold
for other properties of $G$, such as solvability.
Question~\ref{q.smale} could be generalized in the following way.
\begin{ques}
Fix a reduced word $w(f,g_1,\ldots, g_k)$ in $\Diff^1(M)$.  How small can the set $\{{\bf g}\in \left(\Diff^1(M)\right)^k: w(f,{\bf g}) = id\}$
be for the generic $f\in \Diff^1(M)$?
\end{ques}

The history of Question~\ref{q.smale}  goes back to the work of N. Kopell \cite{Ko}, who
gave a complete answer for $r\geq 2$ and the circle $M=S^1$: the
set of diffeomorphisms with trivial centralizer contains an open and dense subset of $\Diff^r(S^1)$.
For $r\geq 2$ on higher dimensional manifolds, there are partial results with additional
dynamical assumptions, such as hyperbolicity \cite{PY1, PY2} and partial hyperbolicity \cite{Bu}.
In the $C^1$ setting, Togawa proved that generic Axiom A diffeomorphisms have trivial centralizer.  In an earlier work
\cite{BCW}, we showed that for $\dim(M)\geq 2$,  the $C^1$ generic conservative (volume-preserving or symplectic)
diffeomorphism has trivial centralizer in $\Diff^1(M)$.  A more precise list of previous results can be found
in \cite{BCW}.

The rest of the paper describes some of the main novelties in the proof of Theorem~\ref{t.main1} and the structure
of its proof.

\subsection*{Local and global: the structure of the proof of Theorem~\ref{t.main1}}
The proof of Theorem~\ref{t.main1} breaks into two parts, a
``local'' one and a ``global'' one. The local part proves that for
the generic $f$,  if $g$ commutes with $f$, then $g=f^\alpha$ on an
open and dense subset  $W\subset M$, where $\alpha\colon W\to \ZZ$
is a locally constant function. The global part consists in proving
that for generic $f$, $\alpha$ is constant. This is also the general
structure of the proofs of the main results in \cite{Ko, PY1, PY2,
To1, To2, Bu2}.  In contrast, in the context of the $C^1$ flow
embedding problem studied by J. Palis \cite{P}, there are local
obstructions, like the existence of transverse heteroclinic orbits,
which prevent a diffeomorphism from being embedded in a flow.

\renewcommand{\thesubsubsection}{\alph{subsubsection})}
\subsubsection{The local strategy}

In describing the local strategy, let us first make a very rough analogy with the symmetries of a
Riemanniann manifold. If you want to prevent a Riemanniann metric from having global isometries,
it is enough to perturb the metric in order to get a point which is locally isometric to no others,
and which does not admit any local isometries. Hence the answer to the global problem is indeed
given by a purely local perturbation, and the same happens
for the flow embedding problem: if a diffeomorphism $f$ does not agree in some place with
the time-1 map of a flow, then neither does the global diffeomorphism.

The situation
of the centralizer problem is quite different: the centralizer of $f$ may be locally trivial at
some place, but $f$ may still admit a large centralizer supported in another place.
Coming back to our analogy with isometries, our strategy consists in producing local perturbations covering a open and dense
subset of orbits, avoiding non-trivial local symmetries on that set. This step consists in ``individualizing'' a dense collection
of orbits,  arranging that the behavior of the diffeomorphism in a neighborhood of one orbit is different from the behavior in a neighborhood
of any other. Hence any commuting diffeomorphism must preserve each of these orbits.

This individualization of orbits happens whenever a property of
unbounded distortion (UD) holds between certain orbits of $f$, a
property which we describe precisely in the next section. In the
first step of our proof we show that the (UD) property holds for a
residual set of $f$.  This gives local rigidity of the centralizer
of a generic $f$, which gives the locally constant function
$\alpha$.

\subsubsection{The global strategy}

The global strategy goes like this. Assuming that we already proved the first step, we have that any diffeomorphism $g$ commuting with the
generic $f$ is on the form $g=f^\alpha$ where $\alpha$ is locally constant and defined on a dense open subset. Furthermore,
$\alpha$ is uniquely defined on the non-periodic points for $f$.
Assuming that the periodic points of $f$ are isolated, it is now enough to verify that
the function $\alpha$ is bounded. This would be the case if the derivative $Df^n$ takes
large values on each orbit of $f$, {\em for each large $n$}: the bound on $Dg$ would
then forbid $\alpha$ from taking arbitrarily large values.
Notice that this property is global in nature: we require large derivative of $f^n$ on each orbit, for each large $n$.

Because it holds for every orbit (not just a dense set of orbits)
and every large $n$, this large derivative (LD) property is not
generic, although we prove that it is dense. This lack of genericity
affects the structure of our proof: it is not possible to obtain
both (UD) and (LD) properties just by intersecting two residual
sets.  There are two more steps in the argument.  First, we show
that {\em among the diffeomorphisms satisfying (UD)}, the property
(LD) is dense.  This allows us to conclude that the set of
diffeomorphisms with trivial centralizer is $C^1$-dense, answering
the first part of Question~\ref{q.smale}.

\subsubsection{From dense to residual}
At this point in the proof, we have obtained a $C^1$-dense set of
diffeomorphisms with trivial centralizer. There is some subtlety in
how we obtain a residual subset from a dense subset. An obvious way
to do this would be to prove that the set of diffeomorphisms with
trivial centralizer form a $G_\delta$, i.e., a countable
intersection of open sets. It is not however clear from the
definition that this set is even a Borel set, let alone a
$G_\delta$.  Instead we use a semicontinuity argument. To make this
argument work, we must consider centralizers defined inside of a
larger space of homeomorphisms, the bi-Lipschitz homeomorphisms. The
compactness of the space of bi-Lipschitz homeomorphisms with bounded
norm is used in a crucial way.  The details are described below. The
conclusion is that if a $C^1$-dense set of diffeomorphisms has
trivial centralizer inside of the space of bi-Lipschitz
homeomorphisms, then this property holds on a $C^1$ residual set.

\section{Background on $C^1$-generic dynamics}\label{s.preliminaries}
The space $\Diff^1(M)$ is a Baire space in the $C^1$ topology. A
{\em residual} subset of a Baire space is one that contains a
countable intersection of open-dense sets; the Baire category
theorem implies that a residual set is dense.  We say that a
property holds for the {\em $C^1$-generic diffeomorphism} if it holds
on a residual subset of $\Diff^1(M)$.

For example, the Kupka-Smale Theorem asserts (in part)
that for a $C^1$-generic diffeomorphism $f$, the periodic orbits of 
$f$ are all hyperbolic. It is easy to verify that, furthermore, 
the $C^1$-generic diffeomorphism $f$ has the following property:
if $x,y$ are periodic points of $f$ with period
$m$ and $n$ respectively, and if their orbits are distinct, 
then the set of eigenvalues of $ Df^m(x)$ and of $Df^n(y)$ are disjoint. 
If this property holds, we
say that the \emph{periodic orbits of $f$ have distinct
eigenvalues}.


The nonwandering set  $\Omega(f)$ is the set of all points $x$
such that every neighborhood $U$ of $x$ meets some iterate of $U$:
$$U\cap \bigcup_{k>0} f^k(U) \ne \emptyset.$$
The elements of $\Omega(f)$ are called {\em nonwandering points}. By
the canonical nature of its construction, the compact set
$\Omega(f)$ is preserved by any homeomorphism $g$ that commutes with
$f$.

%
%

In \cite{BC} it is shown that for a $C^1$-generic diffeomorphism $f$, each
connected component $O$ of the interior of $\Omega(f)$ is contained
in the closure of the stable manifold of a periodic point $p\in O$.
Conceptually, this result means that for $C^1$ generic $f$, the
interior of $\Omega(f)$ and the {\em wandering set} $\M\setminus \Omega(f)$
share certain nonrecurrent features, as we now explain.

While points in the interior of $\Omega(f)$ all have nonwandering dynamics,
if one instead considers the restriction of $f$ to a stable manifold of a periodic orbit
$W^s(p)\setminus \cO(p)$, the dynamics are no longer recurrent; in the
induced topology on the submanifold $W^s(p)\setminus \cO(p)$, every point
has a {\em wandering neighborhood} $V$ whose iterates are all disjoint from $V$.
Furthermore, the sufficiently large future iterates of such a 
wandering neighborhood are contained in a neighborhood of a periodic orbit.
While the forward dynamics
on the wandering set are not similarly ``localized'' as they are on a stable manifold,
they still share this first feature: on the wandering set, 
every point has a wandering neighborhood (this time
the neighborhood is in the topology on $M$).

Thus, the results in \cite{BC} imply
that for the $C^1$ generic $f$, we have the following picture: 
there is an $f$-invariant open and dense subset $W$ of $M$, consisting 
of the union of the interior of $\Omega(f)$ and the complement 
of $\Omega(f)$, and densely in $W$ the dynamics of $f$ 
can be decomposed into components with ``wandering strata.''  
We exploit this fact
in our local strategy, outlined in the next section.


\section{Conditions for the local strategy: the unbounded distortion (UD) properties}

In the local strategy, we control the dynamics of the $C^1$ generic $f$ on the
open and dense set $W=\interior(\Omega(f))\cup \left(M\setminus\Omega(f)\right)$.
We describe here the main analytic properties we use to
control these dynamics.

We say that diffeomorphism $f$ satisfies the \emph{unbounded distortion
property  on the wandering set (UD$^{M\setminus\Omega}$)} if there
exists a dense subset $\cX\subset M\setminus \Omega(f)$ such that,
for any $K>0$, any $x\in \cX$ and any $y\in M\setminus \Omega(f)$ 
not in the orbit of $x$,  there exists $n\geq 1$ such that:
$$|\log |\det Df^n(x)|-\log|\det Df^n(y)||>K.$$

A diffeomorphism $f$ satisfies the \emph{unbounded distortion
property  on the stable manifolds(UD$^s$)} if for any hyperbolic
periodic orbit $\cO$, there exists a dense subset $\cX\subset
W^s(\cO)$ such that, for any $K>0$, any $x\in \cX$ and any $y\in W^s(\cO)$
not in the orbit of $x$, there exists $n\geq 1$ such that:
$$|\log |\Det Df_{|W^s(\cO)}^n(x)|-\log|\Det Df_{|W^s(\cO)}^n(y)||>K.$$

Our first main perturbation result in \cite{BCW2} is:
\begin{theo}[Unbounded distortion]\label{t.UD}
The diffeomorphisms in a residual subset of $\diff^ 1(M)$ satisfy
the (UD$^{M\setminus\Omega}$) and  the (UD$^s$) properties.
\end{theo}

A variation of an argument due to Togawa \cite{To1,To2} detailed in 
\cite{BCW} shows the (UD$^s$) property holds for a $C^1$-generic diffeomorphism. 
To prove Theorem~\ref{t.UD}, we
are thus left to prove that the (UD$^{M\setminus\Omega}$) property holds
for a $C^1$-generic diffeomorphism.  This property is significantly
more difficult to establish $C^1$-generically than the (UD$^s$) property.
The reason is that points on the stable manifold of a periodic
point all have the same future dynamics, and these dynamics are
``constant'' for all large iterates: in a neighborhood of the periodic
orbit, the dynamics of $f$ are effectively linear.  In the wandering
set, by contrast, the orbits of distinct points can be completely unrelated
after sufficiently many iterates.

Nonetheless, the proofs that the (UD$^{M\setminus\Omega}$) and (UD$^s$) properties
are $C^1$ residual share some essential features, and both rely on the essentially non-recurrent
aspects of the dynamics on both the wandering set and the stable manifolds.

\section{Condition for the global strategy: the large derivative (LD) property}

Here we describe the analytic condition on the $C^1$-generic $f$ we
use to extend the local conclusion on the centralizer of $f$ to a global conclusion.

A diffeomorphism $f$ satisfies the \emph{large derivative property (LD) on a set $X$}
if, for any $K>0$, there exists $n(K)\geq 1$ such that
for any $x\in X$ and $n\geq n(K)$, there exists
$j\in \ZZ$ such that:
$$\sup\{\|Df^n(f^j(x))\|,\|Df^{-n}(f^{j+n}(x))\|\}>K.$$
Rephrased informally, the (LD) property on $X$ means that 
the derivative $Df^n$ ``tends to $\infty$'' {\em uniformly} on 
all orbits passing through $X$.  
We emphasize that the large derivative property is a property of the
{\em orbits} of points in $X$, and if it holds for $X$, it also
holds for all iterates of $X$.

The second main perturbation result in \cite{BCW2} is:
\begin{theo}[Large derivative]\label{t.LD}
Let $f$ be a diffeomorphism whose periodic orbits are hyperbolic.
Then, there exists a diffeomorphism $g$ arbitrarily close to $f$ in $\diff^1(M)$
such that the property (LD) is satisfied on $M\setminus \Per(f)$.

Moreover,
\begin{itemize}
\item $f$ and $g$ are conjugate via a homeomorphism $\Phi$, i.e. $g=\Phi f\Phi^{-1}$;
\item for any periodic orbit $\cO$ of $f$, the derivatives of $f$ on $\cO$
and of $g$ on $\Phi(\cO)$ are conjugate (in particular the periodic orbits of $g$ are hyperbolic);
\item if $f$ satisfies the (UD$^{M\setminus\Omega}$) property, then so does $g$;
\item if $f$ satisfies the (UD$^s$) property, then so does $g$.
\bigskip\\
\end{itemize}
\end{theo}

As a consequence of Theorems~\ref{t.UD} and \ref{t.LD} we obtain:
\begin{coro}\label{c.UDLD}
There exists a dense subset $\cD$ of $\diff^1(M)$
such that any $f\in \cD$ satisfies the following properties:
\begin{itemize}
\item the periodic orbits are hyperbolic and have distinct eigenvalues;
\item any component $O$ of  the interior of $\Omega(f)$ contains a periodic point whose stable manifold is dense in $O$;
\item $f$ has the (UD$^{M\setminus\Omega}$) and the (UD$^s$) properties;
\item $f$ has the (LD) property on $M\setminus \Per(g)$.
\end{itemize}
\end{coro}

The proofs of Theorems~\ref{t.UD} and \ref{t.LD} are intricate, incorporating
the topological towers developed in \cite{BC} with novel perturbation techniques.
We say more about the proofs in Section~\ref{s.conclusion}.

\section{Checking that the centralizer is trivial}
We now explain why properties (UD) and (LD) together imply that the centralizer is trivial.
\begin{proposition}\label{p=denseC1}
Any diffeomorphism $f$ in the $C^1$-dense subset $\cD\subset\Diff^1(M)$ given by Corollary~\ref{c.UDLD}
has a trivial centralizer $Z^1(f)$.
\end{proposition}

\begin{proof}[Proof of Proposition~\ref{p=denseC1}]
Consider a diffeomorphism $f\in \cD$.
Let $g\in Z^{1}(f)$ be a diffeomorphism commuting with $f$, 
and let $K>0$ be a Lipschitz constant for $g$ and $g^{-1}$.
Let $W = \interior(\Omega(f))\cup\left(M\setminus \Omega(f) \right)$
be the $f$-invariant, open and dense subset of $M$ whose properties are discussed in Section~\ref{s.preliminaries}.

Our first step is to use the ``local hypotheses'' (UD$^{M\setminus\Omega}$)
and (UD$^s$) to construct a function $\alpha\colon W \to\ZZ$ that is constant on each connected component
of $W$ and satisfies $g=f^\alpha$. We then use the ``global hypothesis'' (LD) to show that $\alpha$ is 
bounded on $W$, and therefore extends to a constant function on $M$.

We first contruct $\alpha$ on the wandering set $M\setminus\Omega(f)$.
The basic properties of Lipschitz functions and the relation $f^n g = g f^n$ imply that
 for any $x\in M$, and any $n\in \ZZ$, we have
\begin{eqnarray}\label{e=bounded}
|\log \det(Df^n(x))-\log \det(Df^n(g(x)))|\leq 2d \log K,
\end{eqnarray}
where $d=\dim M$. On the other hand, $f$ satisfies the
UD$^{M\setminus\Omega(f)}$ property, and hence there is dense subset
$\cX \subset M\setminus\Omega(f)$, each of whose points has unbounded
distortion with respect to any point in the wandering set not on the
same orbit. That is, for any $x\in \cX$, and $y\in M\setminus \Omega(f)$
not on the orbit of $x$, we have:
$$\limsup_{n\to \infty} |\log |\det Df^n(x)|-\log|\det Df^n(y)|| = \infty.$$
Inequality (\ref{e=bounded}) then implies that $x$ and $y=g(x)$ lie on
the same orbit, for all $x\in\cX$, hence $g(x)=f^{\alpha(x)}(x)$.
Using the continuity of $g$ and the fact that the points
in $M\setminus\Omega(f)$ admit wandering neighborhoods whose  $f$-iterates
are pairwise disjoint, we deduce that the map  $\alpha\colon \cX\to \ZZ$ is
constant in the neighborhood of any point in $M\setminus\Omega(f)$.
Hence the function $\alpha$ extends on $M\setminus\Omega(f)$ to a
function that is constant on each connected component of
$M\setminus\Omega(f)$. Furthermore,  $g=f^\alpha$ on
$M\setminus\Omega(f)$.

We now define the function $\alpha$ on the interior $\interior(\Omega(f))$ of the nonwandering set.
The hypotheses on $f$ imply that each component of $\interior(\Omega(f))$ contains
a dense stable manifold of a periodic point.  Hence it suffices to prove the existence
of such an $\alpha$ on the stable manifolds of periodic orbits.
Since the periodic orbits of $f\in \cD$ have distinct
eigenvalues, the diffeomorphism $g$ preserves each periodic orbit of $f$.

We then use the fact that $f\in\cD$ satisfies the $UD^s$ condition.
As noted in Section~\ref{s.preliminaries}, for every periodic point $p$ 
of $f$ the points in $W^s(p)\setminus
\{p\}$ are wandering for the restriction of $f$ to $W^s(p)$.  Hence,
arguing as above, we obtain that for any periodic point $p$, the
diffeomorphism $g$ coincides with a power $f^\alpha$ on each connected
component of $W^s(p)\setminus \{p\}$. For $f\in\cD$, each connected
component $O$ of the interior of $\Omega(f)$ contains a periodic
point $x$ whose stable manifold is dense in $O$. One deduces that
$g$ coincides with some power $f^\alpha$ of $f$ on each connected
component of the interior of $\Omega(f)$.

We have seen that there is a locally constant function $\alpha\colon W\to\ZZ$ 
such that $g=f^\alpha$ on the $f$ invariant, open and dense subset $W\subset M$. 
We now turn to the global strategy.
Notice that, since
$f$ and $g$ commute, the function $\alpha$ is constant along the
orbits of $f$. Now $f\in\cD$ satisfies the (LD) property. 
Consequently there exists $N>0$ such that, for every non-periodic point
$x$, and for every $n\geq N$ there is a point $y=f^i(x)$ such that
either $\|Df^n(y)\|>K$ or $\|Df^{-n}(y)\|>K$. This implies that the
function $|\alpha|$ is bounded by $N$: otherwise, $\alpha$ would be
greater than $N$ on the invariant open set $W$ of $M$. This open set
contains a non-periodic point $x$ and an iterate $y=f^i(x)$ such
that either $\|Df^\alpha(y)\|>K$ or $\|Df^{-\alpha}(y)\|>K$. This
contradicts the fact that $g$ and $g^{-1}$ are $K$-Lipschitz.

We just showed that $|\alpha|$ is bounded by some integer $N$.
Let $\Per_{2N}$ be the set of periodic points of $f$ whose period is less than $2N$
and for $i\in \{-N,\dots,N\}$ consider the set
$$P_i=\{x\in M\setminus \Per_{2N},\; g(x)=f^i(x)\}.$$
This is a closed invariant subset of $M\setminus \Per_{2N}$.
What we proved above implies that $M\setminus \Per_{2N}$
is the union of the sets $P_i$, $|i|\leq N$.
Moreover any two sets $P_i,P_j$ with $i\neq j$ are disjoint since a point in $P_i\cap P_j$
would be $|i-j|$ periodic for $f$.
Since $M\setminus \Per_{2N}$ is connected, one deduces that only one set $P_i$ is non-empty,
implying that $g=f^i$ on $M$.
\end{proof}

\section{From dense to residual: compactness and semicontinuity}
The previous results show that the set of diffeomorphisms having a trivial centralizer
is dense in $\diff^1(M)$ but it is not enough to conclude the proof of Theorem~\ref{t.main1}.
Indeed the dense subset $\cD$ in Theorem~\ref{t=liptriv} is {\em not} a
residual subset if $\dim(M)\geq 2$. (In the final version of this
work we will provide a non-empty open set in which $C^1$-generic
diffeomorphisms does not satisfy the (LD)-property).

Fix a metric structure on $M$. A homeomorphism $f:M\to M$ is {\em
$K$-bi-Lipschitz} if both $f$ and $f^{-1}$ are Lipschitz, with
Lipschitz norm bounded by $K$.  A homeomorphism that is
$K$-bi-Lipschitz for some $K$ is called a  {\em bi-Lipschitz
homeomorphism}, or {\em lipeomorphism}.  We denote by $\Lip^K(M)$
the set of $K$-bi-Lipschitz homeomorphisms of $M$ and by $\Lip(M)$
the set of bi-Lipschitz homeomorphisms of $M$. The Arz{\`e}la-Ascoli
theorem implies that $\Lip^K(M)$ is compact in the uniform  ($C^0$)
topology. Note that $\Lip(M)\supset \Diff^1(M)$.

For $f\in \Lip(M)$, the set $Z^{Lip}(f)$ is defined analogously to the $C^r$ case:
$$Z^{Lip}(f):=\{g\in \Lip(M): fg=gf\}.$$
Now Theorem~\ref{t.main1} is a direct corollary of:
\begin{theo}\label{t=liptriv}
The set of diffeomorphisms $f$ with trivial centralizer $Z^{lip}(f)$ is residual in $\Diff^1(M)$.
\end{theo}
The proof of Theorem~\ref{t=liptriv} has two parts.

\begin{prop}\label{p=denselip}
Any diffeomorphism $f$ in the $C^1$-dense subset $\cD\subset\Diff^1(M)$ given by Corollary~\ref{c.UDLD}
has a trivial centralizer $Z^{lip}(f)$.
\end{prop}
The proof of this proposition is the same as for Proposition~\ref{p=denseC1}.

\begin{prop}\label{p=liptocr}
Consider the set $\cT$ of diffeomorphisms $f\in \Diff^1(M)$ having a trivial centralizer $Z^{lip}(f)$.
Then, if $\cT$ is dense in $\Diff^1(M)$, it is also residual.
\end{prop}

\begin{remarkempty}
Proposition~\ref{p=liptocr} also holds in the $C^r$ topology
$r\geq 2$ on any manifold $M$ on which the $C^r$-generic
diffeomorphism has at least one hyperbolic periodic orbit (for example, on
the circle, or on manifolds of nonzero Euler characteristic).
On the other hand, Theorem~\ref{t=liptriv} is false in general in the $C^2$ topology.
In fact, a simple folklore argument (see the proof of Theorem B in~\cite{navas}) 
implies that for {\em any} Kupka-Smale diffeomorphism
$f\in \Diff^2(S^1)$, the set $Z^{Lip}(f)$ is {\em infinite dimensional}.
It would be interesting to find out what is true in higher dimensions.
\end{remarkempty}

\begin{proof}[Proof of Proposition~\ref{p=liptocr}]
For any compact metric space $X$ we denote
by $\cK(X)$ the set of non-empty compact subsets of $X$ in the Hausdorff
topology, endowed with the Hausdorff distance $d_H$.
We use the following classical fact. 

\begin{proposition}\label{p=semi} Let $\cB$ be a Baire space,
let $X$ be a compact metric space, and let $h:\cB\to \cK(X)$ be an
upper-semicontinuous function. Then the set of continuity points of
$h$ is a residual subset of $\cB$.

In other words, if $h$ has the property that for all $b\in \cB$,
$$b_n\to b\,\implies\, \limsup b_n= \bigcap_n\overline{\bigcup_{i>n} h(b_i)} \subseteq h(b),$$
then there is a residual set $\cR_h\subset \cB$ such that, for all
$b\in \cR_h$,
$$b_n\to b\implies\, \lim d_H(b_n,b)=0.$$
\end{proposition}

To prove Proposition~\ref{p=liptocr}, we note that for a fixed $K>0$,
the set $Z^{Lip}(f)\cap Lip^K(M)$ is a closed subset (in the $C^0$
topology) of the compact metric space $Lip^K(M)$.  This is a simple
consequence of the facts that $Z^{lip}(f)$ is defined by the
relation $fgf^{-1}g^{-1} = id$, and that composition and inversion
are continuous. Thus there is well-defined map $h_K$ from
$\Diff^1(M)$ to $\cK(Lip^K(M))$, sending $f$ to $h_K(f) =
Z^{Lip}(f)\cap Lip^K(M)$. It is easy to see that $h_K$ is
upper-semicontinuous: if $f_n$ converges to $f$ in $\diff^1(M)$ and
$g_n\in h_K(f_n)$ converges uniformly to $g$ then $g$ belongs to
$h_K(f)$.

Let $\cR_K\subset \Diff^1(M)$
be the set of points of continuity of $h_K$; it is
a residual subset of $\Diff^1(M)$, by Proposition~\ref{p=semi}.  Let
$\cR_{Hyp}\subset \Diff^1(M)$ be the set of diffeomorphisms
such that each $f\in \cR_{Hyp}$ has at least one
hyperbolic periodic orbit (the $C^1$ Closing Lemma implies
that $\cR_{Hyp}$ is residual).  Finally, let
$$\cR = \cR_{Hyp}\cap \bigcap_{K=1}^{\infty} \cR_K.$$

Assuming that $\cT$ is dense in $\Diff^1(M)$,
we claim that the set $\cR$ is contained in $\cT$ implying that $\cT$ is residual.
To see this, fix $f\in \cR$, and let $f_n\to f$ be a sequence of
diffeomorphisms in $\cT$ converging to $f$ in the $C^1$ topology.
Let $g\in Z^{Lip}(M)$ be a $K$-bi-Lipschitz homeomorphism satisfying
$fg=gf$.  Since $h_K$ is continuous at $f$, there is a sequence
$g_n\in Z^{Lip}(f_n)$ of $K$-bi-Lipschitz homeomorphisms
with $g_n\to g$ in the $C^0$ topology.
The fact that $f_n\in\cT$ implies that
the centralizer $Z^{Lip}(f_n)$ is trivial,  so
there exist integers $m_n$ such that $g_n = f^{m_n}$.

If the sequence $(m_n)$ is bounded, then passing to a subsequence,
we obtain that $g = f^m$, for some integer $m$.
If  the sequence $(m_n)$ is
not bounded, then we obtain a contradiction as follows.  Let $x$ be a
hyperbolic periodic point of $f$, of period $p$.
For $n$ large, the map $f_n$ has a periodic orbit $x_n$ of period $p$,
and the derivatives $Df^p_n(x_n)$ tend to the derivative $Df^p(x)$.
But then $|\log\|Df_n^{m_n}\||$ tends to infinity as $n\to\infty$.
This contradicts the fact that the diffeomorphisms
$f^{m_n}_n=g_n$ and $f^{-m_n}_n=g_n^{-1}$ are both $K$-Lipschitz,
concluding the proof.
\end{proof}

\section{Conclusion}\label{s.conclusion} To complete the proof of
Theorem~\ref{t.main1}, it remains to prove Theorems~\ref{t.UD}
and \ref{t.LD}. Both of these results split in two parts.  The first
part is a local perturbation tool, which changes the derivative of $f$ in a very small
neighborhood of  a point, the neighborhood being chosen so small
that $f$ looks like a linear map on many iterates of this
neighborhood. In the second part, we perform perturbations provided by the
first part at different places in such a
way that the derivative of every (wandering or non-periodic) orbit
will be changed in the desirable way. For the (UD) property on the
wandering set, the existence of open sets disjoint from all its
iterates are very helpful, allowing us to spread the perturbation
out over time. For the (LD) property, we need to control every
non-periodic orbit. The existence of \emph{topological towers} with
very large return time,  constructed in \cite{BC}, are the main tool,
allowing us again to spread the perturbations out over a long time
interval.

\begin{thank}
We thank Andres Navas for calling our attention to the Lipschitz centralizer
and for pointing out to us Ghys's paper~\cite{ghys}.
\end{thank}

\vspace{10pt}

\noindent \textbf{Christian Bonatti (bonatti@u-bourgogne.fr)}\\
\noindent  CNRS - Institut de Math\'ematiques de Bourgogne, UMR 5584\\
\noindent  BP 47 870\\
\noindent  21078 Dijon Cedex, France\\
\vspace{10pt}

\noindent \textbf{Sylvain Crovisier (crovisie@math.univ-paris13.fr)}\\
\noindent CNRS - Laboratoire Analyse, G\'eom\'etrie et Applications, UMR 7539,\\
\noindent Institut Galil\'ee, Universit\'e Paris 13, Avenue J.-B. Cl\'ement,\\
\noindent 93430 Villetaneuse, France\\
\vspace{10pt}

\noindent \textbf{Amie Wilkinson (wilkinso@math.northwestern.edu)}\\
\noindent Department of Mathematics, Northwestern University\\
\noindent 2033 Sheridan Road \\
\noindent Evanston, IL 60208-2730,  USA


\begin{thebibliography}{ABC}



\bibitem[BC]{BC} Bonatti, Ch.; Crovisier, S.,
R\'ecurrence et g\'en\'ericit\'e.
{\em Invent. Math.} {\bf 158} (2004), 33--104.

\bibitem[BCW1]{BCW} Bonatti, Ch.; Crovisier, S.; Wilkinson, A.,
{\em $C^1$-generic conservative diffeomorphisms have trivial centralizer.} In preparation.\\
A previous version was: {\em Centralizers of $C^1$-generic diffeomorphisms.}
Preprint (2006) arXiv:math/0610064.

\bibitem[BCW2]{BCW2} Bonatti, Ch.; Crovisier, S.; Wilkinson, A.,
{\em $C^1$-generic diffeomorphisms have trivial centralizer.} In preparation.

\bibitem[BCVW]{BCVW} Bonatti, Ch.; Crovisier, S.; Vago, G; Wilkinson, A.,
{\em Local density of diffeomorphisms large centralizers.} In preparation.


\bibitem[Bu1]{Bu} Burslem, L.,
Centralizers of partially hyperbolic diffeomorphisms.
{\em Ergod. Th. \& Dynam. Sys.} {\bf 24}  (2004),  no. 1, 55--87.

\bibitem[Bu2]{Bu2} Burslem, L., 
Centralizers of area preserving diffeomorphisms on $S\sp 2$.
{\em Proc. Amer. Math. Soc.}  {\bf 133 } (2005),  no. 4, 1101--1108.

\bibitem[FRW]{FRW} Foreman, M.; Rudolph, D.; Weiss, L.,
On the conjugacy relation in ergodic theory.
{\em C. R. Math. Acad. Sci. Paris}  {\bf 343} (2006), 653--656.

\bibitem[G]{ghys} Ghys, \'E.,
Groups acting on the circle.
{\em L'Enseign. Math.} \textbf{47} (2001), 329--407.

\bibitem[Ko]{Ko} Kopell, N.,
Commuting diffeomorphisms. In {\em Global Analysis},
Proc. Sympos. Pure Math., Vol. XIV, AMS (1970), 165--184.


\bibitem[N]{navas} Navas, A.,
{\em Three remarks on one dimensional bi-Lipschitz conjugacies.}
Preprint 2007, arXiv:0705.0034.

\bibitem[P]{P} Palis, J.,
Vector fields generate few diffeomorphisms.
{\em Bull. Amer. Math. Soc.} {\bf 80} (1974), 503--505. 

\bibitem[PY1]{PY1} Palis, J.; Yoccoz, J.-C.,
Rigidity of centralizers of diffeomorphisms.
{\em Ann. Sci. \'Ecole Norm. Sup.} {\bf 22} (1989), 81--98.

\bibitem[PY2]{PY2} Palis, J.; Yoccoz, J.-C.,
Centralizers of Anosov diffeomorphisms on tori.
{\em Ann. Sci. \'Ecole Norm. Sup.} {\bf 22} (1989), 99--108.


\bibitem[Sm1]{Sm1} Smale, S.,
Dynamics retrospective: great problems, attempts that failed.
{\em Nonlinear science: the next decade} (Los Alamos, NM, 1990).
Phys. D  {\bf 51}  (1991), 267--273.

\bibitem[Sm2]{Sm2} Smale, S.,
Mathematical problems for the next century.
{\em Math. Intelligencer} {\bf 20} (1998), 7--15.

\bibitem[To1]{To1} Togawa, Y.,
Generic Morse-Smale diffeomorphisms have only trivial symmetries.
{\em Proc. Amer. Math. Soc.} {\bf 65} (1977), 145--149.

\bibitem[To2]{To2} Togawa, Y.,
Centralizers of $C\sp{1}$-diffeomorphisms.
{\em Proc. Amer. Math. Soc.} {\bf 71} (1978), 289--293.
\end{thebibliography}
\end{document}